\documentclass[a4paper,11pt]{amsart}
\usepackage{amsmath,amssymb,amsfonts}
\usepackage{graphicx}
\usepackage{xcolor}

\newtheorem{theorem}{Theorem}
\newtheorem{lemma}[theorem]{Lemma}

\newtheorem{problem}{Problem}

\newtheorem{defin}{Definition}

\begin{document}
\title[Hamilton Cycles in Random Lifts of Graphs]
{Hamilton Cycles in Random Lifts of Graphs}
%
%
\author{ Tomasz \L uczak \and \L ukasz Witkowski \and Marcin Witkowski}
\address{Faculty of  Mathematics and Computer Science,
Adam Mickiewicz University,
ul.~Umultowska 87, 61-614 Pozna\'n, Poland}
\thanks{The first author partially supported by NCN grant 
Maestro, while the third one is partially supported by 
NCN grant Preludium}
\maketitle              

\begin{abstract}
For a graph  $G$ the random $n$-lift of $G$ is obtained by replacing each of its vertices by a set of $n$ vertices, and joining a pair of sets by a random matching whenever the corresponding vertices of $G$ are adjacent. 
We show that asymptotically almost surely the random lift 
of a graph $G$ is hamiltonian, provided $G$ has the minimum degree 
at least $5$ and contains two disjoint Hamiltonian cycles whose union is not a bipartite graph.  
\end{abstract}

\section{Introduction}
The notion of a random lift was proposed by Amit and Linial \cite{pierwsza}
as a discrete version of the topological notion of covering 
maps, which are ``locally bijective'' homomorphisms. 
For graphs $G$ and $H$, a map $\pi:V(H)\rightarrow~V(G)$ 
is a {\em covering map} from $H$ to $G$ 
if for every $v \in V(H)$ the restriction of $\pi$ to 
the neighborhood of $v$ is a bijection onto the neighborhood 
of $\pi(v) \in V(G)$. 
In particular, for every vertex $v \in V(H)$ the degree 
of $v$ must be the same as the degree of $\pi(v)$. 
The set of all vertices which are mapped onto a vertex 
$v$ is called the \textit{fiber above $v$} 
and denoted by $\tilde{G}_v$. 
Since the term covering has been already widely 
used in graph theory, following Amit and Linial, we use the term
{\em lift} instead. For instance, we say that
$H$ from   the previous example is a lift of $G$. 
We often denote the lift of $G$ by $\tilde{G}$.

If for every vertex $v \in G$ 
the fiber $\tilde{G}_v$ has size $n$, 
then we call such a lift an \textit{$n$-lift}. 
We denote the set of all $n$-lifts of a given graph $G$ 
by $L_n(G)$ and call $G$  the \textit{base graph}. 
A random $n$-lift of $G$ is a graph chosen uniformly 
at random from the set $L_n(G)$. This is equivalent 
to associating with each vertex $u \in G$ 
a set $\tilde{G}_u$ of $n$ vertices and 
independently connecting each pair $(\tilde{G}_u,\tilde{G}_v)$ 
by a random matching whenever $u$ and $v$ are adjacent
in the base graph~$G$. Another way to describe 
this process is to take $\tilde{G}_v=\{v_1,...,v_n\}$ 
and $\tilde{G}_u=\{u_1,...,u_n\}$, choose uniformly 
at random one of the $n!$ 
permutations  $\sigma_{vu}:[n]\to [n]$, 
and connect $v_i$ with $u_{\sigma_{vu}(i)}$. 
Note that such permutations (or matching) are chosen independently 
for each edge $uv$ in $G$.

Our interest lies in the asymptotic properties of lifts of graphs, when the parameter $n$ goes to infinity. In particular, we say that a property holds {\em asymptotically almost surely}, or, briefly, {\em aas}, if its  probability tends to $1$ as $n$ tends to infinity. 
Sometimes, instead of 
saying that the random lift of $G$ has aas a property $\mathcal{A}$,
we write  that almost every random lift of a graph $G$ 
has $\mathcal{A}$. 

The first paper in the theory of random lifts of graphs dealt with their connectivity properties. Amit and Linial \cite{pierwsza} have proven that if $G$ is a simple, connected graph with minimum degree $\delta \geq 3$, then its random lift is aas $\delta$-connected. It was shown in \cite{Witkowski} that for graphs $G$ with $\delta(G) \geq 2k -1$ we have an even stronger property, namely a random lift of $G$ is aas $k$-linked. The term \textit{$k$-linked} refers to a graph with the property that for every $2k$ 
distinct vertices $s_1,s_2,\dots,s_k,t_1,t_2,\dots,t_k$ 
the graph contains $k$ vertex-disjoint paths $P_1,P_2,\dots,P_k$ such that $P_i$ connects $s_i$ to $t_i$, $1\leq i \leq k$.

Only a few other properties of random lifts have been studied, 
such as expansion properties \cite{druga}, matchings~\cite{czwarta}, 
and the independence and chromatic numbers~\cite{trzecia}. 
Here we consider the property that a graph contains a Hamilton cycle.
The two main problems concerning Hamiltonicity 
of random lifts, has been stated by Linial \cite{nati2} 
who asked the following two questions.

\begin{problem}
Is it true that for a given  $G$ the random lift 
$L_n(G)$ is either aas hamitonian,  or aas
non Hamiltonian?
\end{problem}
\begin{problem}
Let $G$ be a connected $d$-regular graph with $d \geq 3$. Is it true that random $n$-lift of $G$ is aas hamiltonian?
\end{problem}

Burgin, Chebolu, Cooper and Frieze \cite{frieze} proved the existence of a constant $h_0$, such that if $h \geq h_0$, then graphs chosen uniformly at random from $L(n,K_h)$ and $L(n,K_{h,h})$ are aas hamiltonian. 
Chebolu and Frieze~\cite{fr2} were able to expand this result to  appropriately defined random lifts of complete directed graphs. The main result of this paper goes as follows.

\begin{theorem}\label{thm:main}
Let $G$ be a graph with minimum degree at least five which contains at least two edge 
disjoint Hamilton cycles whose union is not a bipartite graph. Then aas $\tilde{G}\in L_n(G)$ is hamiltonian.
\end{theorem}

This result together with some of its straightforward generalizations covers wide spectrum of graphs. 

The structure of the paper is the following. First we describe general properties of random lifts and the idea behind the algorithm which finds the Hamilton cycle in $\tilde{G}$. Then we present the algorithm. In the next section we show that asymptotically almost surely it succeeds in finding Hamilton cycle in $\tilde{G}$.

\section{Preliminaries}

We start with some general properties of random lifts that will be useful in proving the existence of Hamilton cycle in random lifts.

\begin{lemma}\label{lemma_cycle}
Let $h\ge 3$. Asymptotically almost surely a random $n$-lift of a cycle $C_h$ on $h$ vertices consists of a collection of at most $2 \log n$ disjoint cycles.
\end{lemma}

\begin{proof}
If we remove one edge $e$ from a cycle $C_h$, then we obtain a path. It is easy to see that the lift of the path $P$ is a collection of $n$ disjoint paths. Lifting the missing edge $e$ is the same as matching at random the two sets of ends of those paths or connecting those ends according to some random permutation. The number of cycles created after joining those paths is then the same as the number of cycles in a random permutation on set $[n]=\{1,2,\dots,n\}$. The precise distribution of the number of cycles in random permutation is well known \cite{feller}. In particular aas 
the number of cycles in random permutation is smaller than $2 \log n$.
\end{proof}

Our algorithm will be based on 
the path reversal technique of P\'{o}sa~\cite{Posa}. 
Let $G$ be any connected graph and $P = v_0v_1...v_m$ be a path in $G$. If $1 \leq i \leq m-2$ and $\{v_m,v_i\}$ is an edge of $G$, then $P' = v_0v_1...v_iv_mv_{m-1}...v_{i+1}$ is a path in $G$ with the same vertex set as $P$. We call $P'$ a P\'{o}sa rotation of $P$ with the preserved {\em starting point} $v_0$ and the {\em pivot} $v_i$. 
Note that the edge $\{v_m,v_i\}$ we used for the transformation
into~$P'$ is not incident to its new end. By $\mathcal{P}_{q}(P,v_0)$ we denote the set of all paths of $G$ which can be obtained from $P$ by at most ${q}$ rotations preserving the starting point $v_0$.

Now we briefly comment on the main ideas behind our  algorithm (a more formal description of the procedure we postpone until the next section). For the clarity of argument from this point on symbols that corresponds to elements of base graph will be written in bold. Let $\textbf{G}$ be a connected graph on $k$ vertices with $\delta( \textbf{G} ) \geq 5$ which contains two edge disjoint Hamilton cycles $\textbf{H}_1$ and $\textbf{H}_2$. Choose any vertex $\textbf{h}_1$ and label each vertex twice according to its appearance in Hamilton cycles i.e. $\textbf{H}_1 = \textbf{h}_1\textbf{h}_2\dots \textbf{h}_{k}\textbf{h}_1$ and $\textbf{H}_2 = \textbf{h}'_1\textbf{h}'_2\dots \textbf{h}'_{k}\textbf{h}'_1$, where $\textbf{h}'_1=\textbf{h}_1$. 

Due to Lemma~\ref{lemma_cycle} aas the random lift of $\textbf{H}_1$, denoted by $\tilde{H}_1$, consists of disjoint cycles $C_1,C_2,...,C_\ell$, where $\ell \leq 2\log n$. We refer to these as \textit{basic cycles}. We will use the property that cycles in the lift preserve the order of vertices from the cycles in the base graph, i.e. for every cycle there exists $r\ge 1$ such that it can be written as 
\begin{equation}\label{cycle}
h^1_1h^1_2\dots h^1_{k}h^2_1h^2_2\dots h^2_{k}\dots
h^r_1h^r_2\dots h^r_{k}h^1_1,
\end{equation}
where $h^i_j$ is an element of the fiber $\tilde{G}_{h_j}$. 
 
Let $\textbf{G}_1=\textbf{G}-\textbf{H}_1$. 
Note that $\delta(\tilde{G}_1)=\delta(\textbf{G}_1)\ge 3$.
Our strategy will be rather natural. First we generate the lift $\tilde{H}_1$, next we try to merge cycles  $C_1,C_2,...,C_\ell$ into one long path using edges of $\tilde{G_1}$.  At the end we use the property that $\textbf{G}_1$ contains the Hamilton cycle $\textbf{H}_2$ to close the path into a cycle.

 Denote the longest cycle in $\tilde{H}_1$ by $C$. We shall try to connect $C$ to any other basic cycle in the lift using the edges of $\tilde{G}_1$. Once we succeed in finding connecting edge, we break the cycle $C$ and connect it to other basic cycle, a path created in this way will be denoted by $P$. Subsequently we  want to increase the length of $P$ by ``absorbing'' one basic cycle at a time. We shall do it by generating edges of $\tilde{G}_1$ which are incident to one of the ends of the path $P$. If we connect it to some basic cycle, say $C'_s$, then we replace $P$ by a longer path adding all vertices from $C'_s$, otherwise, either we try to connect the ends of $P$ to create a new cycle $C$, or try to replace $P$ by another path using  P\'osa transformation (in fact, since the probability that we extend $P$ 
is small, we 
use P\'osa transformations right away 
to produce a lot of path with the same vertex set as $P$ 
which became candidates for further extensions).
 If the obtained cycle $C$ is not a Hamilton cycle, then we try to merge $C$ with some of the the remaining basic cycles and repeat the procedure.
   
\section{The algorithm}

We generate a graph $\tilde{G}$ in each step of the algorithm edge by edge. First we generate the lift $\tilde{H}_1$, next at each point we choose for a given vertex $v$ its neighbours in $\tilde{G}_1$ at random one by one from all available candidates. 
Whenever we have already generated an edge from $\tilde{G_1}$ adjacent to a vertex $v$ we call such a vertex \textit{inactive}, vertices that are not inactive are called \textit{active}. We will denote the set of inactive vertices by $D$. 

In the analysis of the algorithm we shall show that asymptotically almost surely in order to merge $P$ with one of the remaining basic cycles we deactivates at most $5n^{4/5}$ vertices. Since at each iteration we connect one basic cycle to the cycle $C$ it would imply that in order to perform the whole procedure, we need to generate edges of $\tilde{G}_1$ incident to not more than $10n^{4/5}\log n  < n^{5/6}$ vertices. In the next chapter we show that in fact at the end of the algorithm execution aas we have $|D| <  n^{5/6}$.
We shall use the fact that $D$ is small quite often in 
order to guarantee that  each time we generate new edges incident with  an active vertex we have a lot of choices, i.e. 
this procedure is not affected much by the fact that some  
edges of $\tilde{G}_1$ have already been revealed. 

The algorithm consists of seven phases.

\smallskip

{\bf Phase 1 -- Cycle Lift} 

{\em Generate a lift $\tilde{H}_1$.  Assign $C$ to be the longest cycle in $\tilde{H}_1$.}

\smallskip

{\bf Phase 2 -- Cycle Merge} 

{\em 
Given a cycle $C$ and a set of basic cycles $C'_1, \dots, C'_s$
disjoint with $C$ do the following:
\begin{itemize}
 \item[A.]{If $0<\sum_{i=1}^s |V(C'_i)|< n^{9/10}$
take any vertex $v$ which belongs to a basic cycle $C'_1$ 
and generate edges of $\tilde{G}_1$ incident to it. If one of these edges $e$ connects $C'_1$ to $C$ assign to $P$ the unique  path 
containing $e$ with  vertex set is $V(C) \cup V(C'_1)$. 
Otherwise take another vertex $v' \in C'_1$ and repeat the operation.}

 \item[B.]{If $\sum_{i=1}^s |V(C'_i)|\ge n^{9/10}$ choose any $n^{1/3}$ active vertices of $C$ which are at distance at least $2$ from 
the set of all inactive vertices and generate edges of $\tilde{G}_1$ 
incident to them. If one of these edges $e$ connects $C'_i$ to $C$
assign to $P$ the unique  path 
containing $e$ with  vertex set is $V(C) \cup V(C'_1)$. 
 Otherwise repeat the operation.}
\end{itemize}
}

\smallskip

{\bf Phase 3 -- Path Merge} 

{\em Given a path $P$ and some basic cycles $C'_1, \dots, C'_s$, 
if any end of $P$ is connected to a basic cycle $C'_{i}$ 
replace $P$ by a new path  with vertex set $V(P)\cup V(C'_{i})$.}

\smallskip

{\bf Phase 4 -- Cloning Path} 

{\em Let us suppose we are given a path $P$, whose both ends are active, and a set of basic cycles $C'_1, \dots, C'_s$. 

Repeat the following actions:

Take $P=w_1w_2\dots w_t$ and apply to it repeatedly P\'{o}sa transformation preserving starting point $w_1$. Continue until $\log^2 n$ different paths starting at $w_1$ and ending at $w_{i_j}$, $j=1,2,\dots, \log^2 n$ will be found.
 Now reverse each of these paths and apply to each of them the transformation preserving point $w_{i_j}$. Continue to perform the operations until one of the conditions is true:

--- there is an edge connecting path $P \in \mathcal{P}_q(P,w_1)$ with some basic cycle $C'_{i_0}$,

--- there is an edge connecting path $P \in \mathcal{P}_{q'}(P_{i_j},w_{i_j})$ with some basic cycle $C'_{i_0}$,

--- there are $r=\log^2 n$ paths $P_1, \dots, P_r$ such that 
each of them has the same vertex set as $P$, and all $2r$ vertices which are ends of these paths are pairwise different and active.

In the case that one of the first two conditions is met go back to Phase~3, in the case that the third condition holds continue to Phase~5.}

\smallskip

{\bf Phase 5 -- Multiplying Ends} 

{\em For every path $P_1, \dots, P_r$ constructed in the Phase 4 split the vertex set $V(P_j)$ of $P_j$  into two roughly equal disjoint sets  $V_1, V_2\subset V(P_j)$, $|V_1|,|V_2|\ge (|V(P_j)|-1)/2$. Thus every path $P_j=w_1w_2 \dots w_t$ splits into two paths $P'_j= w_1w_2\dots w_{i-1}w_i$ and $P''_j=w_{i+1}w_{i+2}\dots w_t$,  where $i=\lceil t/2\rceil$.

At any point of the phase if there is:

--- an edge closing some path $P_j$ to form a cycle, then go to Phase~2,

--- an edge connecting $P_j$ with some basic cycle, then go to  Phase~3.

Repeat simultaneously for each path $P_1, \dots, P_r$: 

 Apply a series of P\'{o}sa transformations to the path $P'_j$ which preserve the starting point $w_{i}$ and a series of P\'{o}sa transformations to the path $P''_j$ which preserve starting point $w_{i+1}$. (We apply a single P\'{o}sa transformation to each of the paths in turn before we apply the next P\'{o}sa transformation).

Stop if for any path you find two sets $S_1\subset V_1$, $S_2\subset V_2$, such that $|S_1|,|S_2|\ge n^{3/5}\log ^2 n$ with the following property:
\begin{quote}
For every $x\in S_1$ and $y\in S_2$ there is a path $P_{xy}$ of length $|P_j|$ which starts at $x$ ends at~$y$ whose first $|V_1|$ vertices are those from $V_1$ and last $|V_2|$ vertices are those from~$V_2$.
\end{quote}
 }
{\bf Phase 6 -- Adjusting} 

{\em  Choose any edge  $\{\textbf{x,y}\}$ from $\textbf{G}\setminus (\textbf{H}_1\cup \textbf{H}_2)$. Use at most $|G|(|S_1|+|S_2|)$ P\'{o}sa transformations to switch the end $w_1$ of the path $P'$ and the end $w_t$ of the path $P''$ to replace the sets $S_1$, $S_2$ generated in the previous stage by slightly smaller sets $S'_1\subset V_1$, $S'_2\subset V_2$, $|S'_1|, |S'_2|\ge n^{3/5}$, such that $S'_1$ is contained in the fiber $\tilde{G}_x$ and $S'_2\subset \tilde{G}_y$. 
}

{\bf Phase 7 -- Closing a cycle}

{\em Generate edges between $\tilde{G}_x$ and $\tilde{G}_y$ incident to vertices from $S'_1$. If one of them has an end in $S'_2$ then  STOP if the resulted cycle is a Hamilton cycle, or otherwise go to Phase 2.}

\section{The analysis of the algorithm}

In this section we show that aas the algorithm returns a Hamiltonian cycle and, consequently, Theorem~\ref{thm:main} follows. We do it by showing that each phase of the algorithm  will be 
successfully completed with probability sufficiently close to~1.

\smallskip

{\bf Phase 1.}
We start the analysis of the algorithm with Phase 1. As already mentioned, Lemma~\ref{lemma_cycle} states that the random lift of $H_1$ asymptotically almost surely consists of disjoint cycles $C_1,C_2,...,C_\ell$, where $\ell \leq 2\log n$. Note that this means that the length of the longest cycle $C \in \tilde{H_1}$ is at least $n/(2\log n)$. Observe that since the number of basic cycles is bounded from above by $2\log n$,  Phases~2 and~3 can be invoked at most $2\log n$ times. 

We shall show that with probability at least $1-o(1/\log n)$ during Phases $2$-$7$ we create a cycle $C$ each time deactivating fewer that $5n^{4/5}$ vertices. 
Thus, at the end of the Algorithm execution we will have 
$|D| \le 10n^{4/5}\log n \leq n^{5/6}$.  This bound, which 
states that to build a cycle we need only to generate a small portion of the random $n$-lift,  shall be used
very often  in the analysis of the algorithm 

Note that in any step in which we deactivate a vertex either it is already in $P$ or we have just added it to $P$. 
Consequently, all vertices outside $P$ are active. Moreover since at every point of the algorithm we want the ends of path $P$ to be active vertices, we perform P\'{o}sa rotation only in the case when the new end of $P$ is an active vertex. Notice that this happen to be true whenever in the rotation the pivot is at distance at least $2$ from any inactive vertex.

{\bf Phase 2.} 
In this step we want to connect cycle $C$ with any basic cycle disjoint with it, creating a long path $P$. As stated above we require that vertices which connect those two cycles are not adjacent to any inactive vertices.

In case A the total number of vertices in the remaining basic cycles 
which are yet to be joined to $C$ is smaller than $n^{0.9}$. The probability that a vertex from the basic cycle $C'_1$ has a neighbour inside $C$ which is at distance at least 2 from any inactive vertex is larger than $$\frac{n-n^{9/10}-\Delta(G)\cdot |D|}{n - |D|} \geq 1-\frac{2n^{9/10}}{0.9n} = 1-o(1/\log n),$$ since we need to exclude vertices outside $C$ together with all inactive vertices and their neighbours. Hence, with probability $1-o(1/\log n)$, the merging deactivates only one vertex and we do not need to repeat the procedure more times. 

For case B note that since $|C|\ge n/(2\log n)$ and there is always fewer than $n^{5/6}$ inactive vertices, one can greedily select $n^{1/3}$ vertices of $C$ which are at distance at least 2 from any inactive vertex and from each other. Clearly, the probability that some of these vertices is adjacent in $\tilde{G}_1$ to one of the basic cycles is bounded from above by 
$$ 1 - \left( \frac{n- n^{9/10}}{n} \right)^{n^{1/3}} \leq 1- \left( 1-n^{-1/10}\right)^{n^{1/3}}=1-o(1/\log n).$$
Again  with probability $1-o(1/\log n)$, we succeed with first set of $n^{1/3}$ vertices and we do not need to repeat the procedure. 

 Altogether each time we invoke this phase 
with probability at least $1-o(1/\log n)$ 
 we deactivate at most $n^{1/3}$ vertices. 

{\bf Phase 3.}  

We do not generate any edges in this step, thus we do not deactivate any vertices. 

{\bf Phase 4.}  
Let $P=w_1,....,w_t$. Our aim is either to find an edge of $\tilde{G}_1$ joining one end of a path $P' \in \mathcal{P}_q(P,w_1)$, or $P'' \in \mathcal{P}_{q'}(P_{i_j},w_{i_j})$, to one of the cycles outside $P$ and go to Phase 3, or to find for $r = \log^2 n$ a set of paths $P_1, \dots, P_r$ such that each of them has the same vertex set as $P$, and all $2r$ vertices which are ends of these paths are different and active.

There are two stages in this phase. First we take path $P$ and find a set of $r$ paths which start at $w_1$ and whose $r$ ends are distinct and active. Notice that after any P\'{o}sa transformation we want the new end to be active so we require that the pivot $w_i$ has no inactive neighbours. Thus we need to estimate the probability that in any of the $\log^2 n$ possibly required P\'{o}sa transformations the new end of our transformed path either is connected to a vertex which is at distance less than $2$ to any inactive vertex, or is the neighbour of one of the ends of previously generated paths. This probability can be crudely bounded above by 

$$\log^2 n \;\frac{(\Delta(G) \cdot |D| + \log^2 n)}{n-|D|} \leq \frac{\Delta(G)\cdot \log^4 n \cdot n^{5/6}}{0.9n} = o(1/\log n)\,.$$

In the second stage we take all paths $P_1,...,P_r$ and apply to them the P\'{o}sa transformations preserving the ends chosen in the first stage. At this time the structure of each path is distinct, so in the process of applying consecutive transformations we might get different results for each path. Moreover we want those new ends to be different from the ends generated in previous stage. Thus we take the first path $P_1$ and apply transformations in order to generate a set of $\log^2 n$ active ends for it and choose one of them as the end of $P_1$. Then we take path $P_2$; if it admits the same transformations as $P_1$, then we select one of the vertices generated for $P_1$, which has not already been taken, as the end for $P_2$. In the opposite case we apply P\'{o}sa transformations for $P_2$ and generate a new set of $\log^2 n$ ends for it. We repeat the same operations for all other paths. Notice that in the worst case scenario we need to make at most $\log^4 n$ single transformations in total. Similarly to the previous case the probability that in any of $\log^4 n$ required P\'{o}sa transformations the new end of our transformed path is 
either connected to a vertex which is within distance $2$ to any inactive vertex, or was the neighbour of the end of one of the previously generated paths, is bounded from above by 

$$\log^4 n \;\frac{(\Delta(G) \cdot |D| + \log^4 n)}{n-|D|} \leq \frac{\Delta(G)\cdot \log^8 n \cdot n^{5/6}}{0.9n} = o(1/\log n)\,.$$

Note also that each time in this phase we have deactivated at most $\log^2 n + \log^4 n \leq 2 \log^4 n$ vertices.

{\bf Phase 5.}   
Let us recall that, roughly speaking, in this phase we want to take any of the paths $P_j~=~w_1w_2\dots w_t$ constructed in the previous case, split it into two halves $P'=w_1w_2\dots w_{i-1}w_i$ and $P''=w_{i+1}w_{i+2}\dots w_t$, where $i=\lceil t/2\rceil$, and apply to them transformations preserving respectively $w_{i}$ and $w_{i+1}$ in order to find at least $n^{3/5}\log^2 n$ new feasible ends for each of them.
 
We show that the probability that we succeed in doing it for 
a given path is bounded away from zero, by some constant $\alpha>0$. Thus if we repeat this for $\log^2 n$ paths, then with probability $1-o(1/\log n)$ for at least one of them we expand the set of feasible ends to the required size.

The existence of a constant $\alpha>0$ follows easily from the theory of branching processes (see \cite{harris}). Indeed, take one path, say $P'$, and first generate all its possible ends using the transformation preserving the end $w_i$ (this will be the first generations of ends), then apply consecutive transformation to obtained ends in order to get the second generations of ends, and so on. In each step we generate at least three new edges (since the minimum degree of $\tilde{G}_1=\tilde G-\tilde H_1$ is at least 
three) and  we fail if we choose in such a trial either a vertex from the other path $P''$, or a vertex which is adjacent to inactive vertex or one of the ends chosen so far. Since $|P''|\le n/2+1\le 0.501n$, the probability of making a bad choice is in each step bounded from above by 
$$\frac{0.501n+\Delta(G)|D|}{n-|D|}\le 0.51.$$
Consequently, the  number of successful choices (i.e. the ones which either lead to a new end or allow us to go to Phase 3) in one round is stochastically bounded from below by the binominally distributed random variable $B(3, 0.49)$. 

Thus, let us recall, we treat the process of applying consecutive 
P\'osa transformations as a branching process. Since every active vertex $v$ has at least $3$ edges in $\tilde{G}_1$ which are still to be revealed, the possible number of descendants for each ancestor is bounded from below by $3$. The probability of producing new individual in the next generation equals the probability that the generated edge connects $v$ with a vertex of $P'$ which neither is  adjacent to an inactive vertex, nor is a vertex generated in previous steps. Since the number 
of inactive vertices is at most $n^{5/6}=o(n)$ 
and clearly $|P''|\le 0.501n$, the process of generating feasible ends for the path $P$ can be 
stochastically bounded from below by the branching process defined by a variable with binomial distribution $B(3,0.49)$. 

 Since $3*0.49>1$, by theory of branching processes \cite{harris} we know that with probability $\beta > 0.61$ the branching process will not die out. Furthermore, a standard large deviation argument (cf. \cite{JLR}, Chapter 5.2) shows that with probability at least $1-2\exp(-n^{3/5})$ the first time we get $n^{3/5}\log^2 n$ vertices in one generation the total number of offspring is bounded from above by $5n^{3/5}\log^2 n$. Consequently, with probability at least $\beta/2$, after using at most $5n^{3/5}\log^2 n$ P\'osa transformations we either merge the end of $P'$ with one of basic cycles (and so go to Phase 3) or generate at least $n^{3/5}\log^2 n$ different active ends for this path. Hence, the probability that it happens at the same time for $P'$ and $P''$ is bounded from below by $\alpha=(\beta/2)^2$. 
  
However,  at the beginning the previous phase of the algorithm provided us not one, but $\log^2 n $ paths with different ends. Consequently, with probability $$1-(1-\alpha)^{\log^2 n}=1-o(1/\log n)$$ we succeed in expanding the set of feasible ends for at least one of the paths. Hence, with probability at least $1-o(1/\log n)$ this phase of the algorithm can be completed with the total number of deactivated vertices bounded from above by 
$\log^2n \cdot 5\Delta(G)n^{3/5}\log^2n\le n^{4/5}$.

{\bf Phase 6.}
The sets $S_1$ and $S_2$ found in the previous phase are such that each edge connecting them 
creates a cycle. Such a cycle is either a Hamilton cycle or can be merged to some remaining basic cycles (back in the Phase $2$). 

Note however that vertices in $S_1$ and $S_2$ are spread over fibers of $\tilde{G}$. 
In particular, it might happen that sets $S_1$ and $S_2$ are placed in two fibers which corresponds to non-adjacent vertices of $\textbf{G}$ and so we cannot expect them to be connected by an edge. 
 Let $\{\textbf{x,y}\} \in E(\textbf{G})$. In this phase we want to use the property that $\textbf{G}_1$ contains second Hamilton cycle $\textbf{H}_2$, to ``switch'' a large part of the sets $S_1$ and $S_2$ 
 to the  fibers $\tilde{G}_x$ and $\tilde{G}_y$, respectively. 

Let $P'=w_1w_{2}\dots w_i$ be defined as ``the half'' of the path we have dealt with in the previous phase, and let $w_1\in S_1$.
We would like to argue that, with probability bounded away from zero by some constant $\gamma>0$, we can deactivate at most $|\textbf{G}|$ vertices in order to either connect $P'$ to some remaining basic cycle, or turn $P'$ by a sequence of Posa transformations preserving the end $w_i$ into a path with the end on the chosen fiber $\tilde{G}_x$. 

Let us recall first that $P'$ has been obtained in the  process of merging
and transforming basic cycles obtained in the first phase. Each of the basic cycles has a periodic structure (see (\ref{cycle})), which implies that they are evenly distributed across the fibers of the lift.
Let $k=|\textbf{G}|$ where, let us recall, $k$ is a constant which does not grow with $n$. In the case when the length of $P'$ is smaller than $n/3$ the total length of basic cycles outside $P'$ and $P''$ is $m\ge n/3$ and furthermore each fiber contains precisely $m/k$ vertices which belong to basic cycles outside $V(P')\cup V(P'')$. Consequently, with positive probability 
(at least $m/(nk)\ge 1/(3k)$) we merge the end of $P'$ with a basic cycle deactivating just one vertex.

Let us consider now the more challenging case, when $P$ is very long and the length of $P'$ is at least $n/3$. We are interested in the structure of the path $P'$, namely to what extent it preserves the structure of basic cycles. Whenever we joined two cycles or performed a P\'osa transformation we perturbed the cyclic distribution of vertices. More precisely, a single merge or transformation could spoil at most three of sequences $...h^i_1h^i_2\dots h^i_{k-1}h^i_kh^{i+1}_1...$ which occur in the path~$P$.  Note that after 
such a transformation the order of the vertices in the sequence is reversed in part of the path (see Figure~1).

\begin{figure}[h!]
\label{rr}
  \centering
    \includegraphics[width=\textwidth]{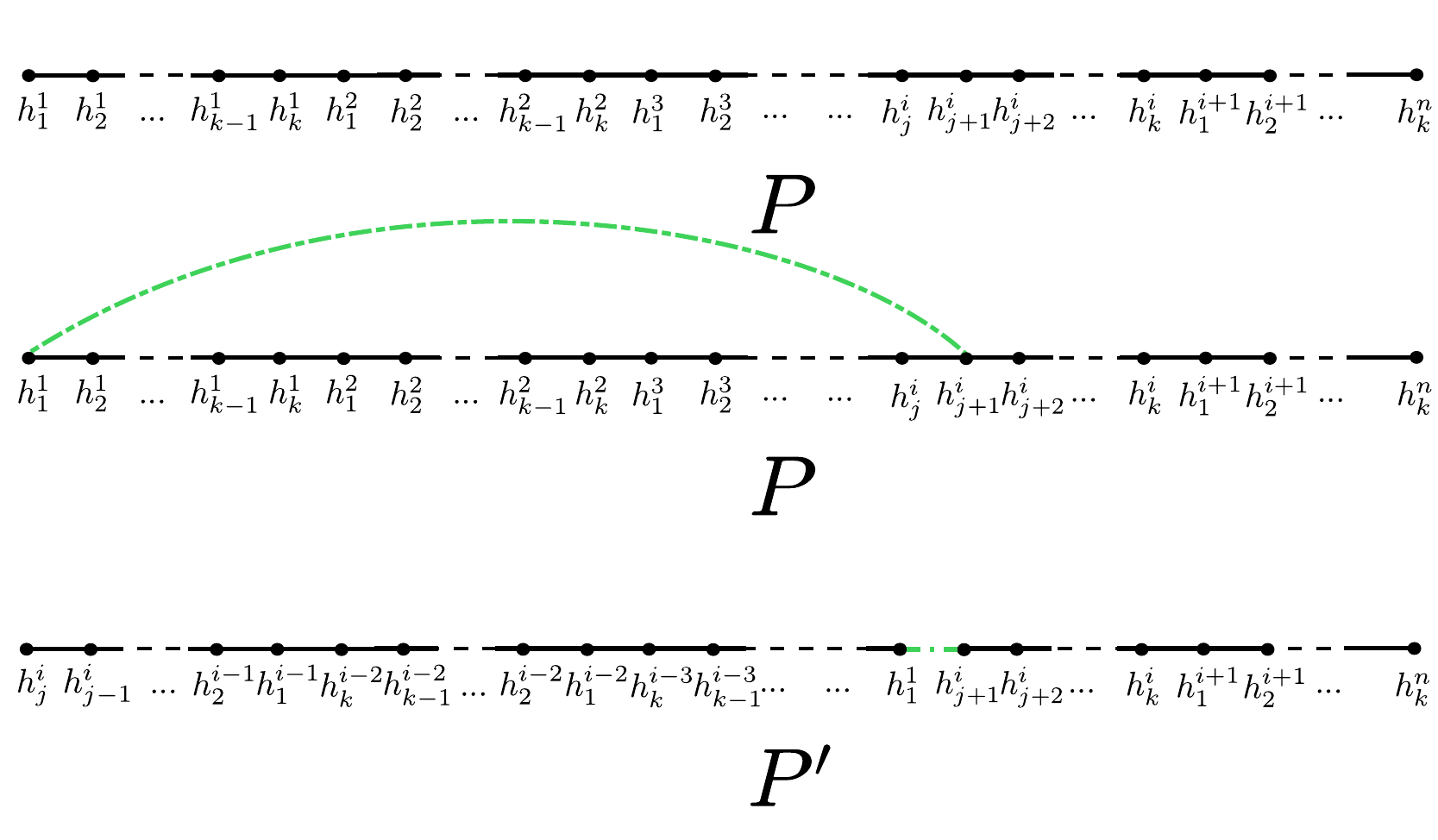}
      \caption{The path $P$ 
			above consists of sequences of vertices from a path which is a part of a lift of Hamilton cycle  
 $\textbf{H}_1 = \textbf{h}_1...\textbf{h}_k$, where by $h_i^j$ we denote a vertex from the fiber above $\textbf{h}_i$. After P\'osa transformation with the pivot $h^i_{j+1}$ we get a new path~$P'$. Notice that $h^i_1...h^i_kh^{i+1}_1$ is no more a sequence of consecutive vertices from fibers above $\textbf{H}_1$. Moreover sequences   $h^1_1...h_1^{i-1}$ are reversed in the path $P'$.}
\end{figure}

 Observe that the number of joins and transformations made to a path $P'$ is bounded by the number of inactive vertices.
Since during the algorithm we deactivate at most $n^{5/6}$ vertices, there are at least $(n/3k)-3n^{5/6}>2n/(7k)$ sequences of consecutive vertices which belong to fibers given by the order of vertices in $\textbf{H}_1$. Some of the sequences could get reversed in the transformations (see Figure~1), but at least half of them, i.e. at least $n/(7k)$, are sequences of consecutive vertices appearing in the order $\textbf{h}_1\dots \textbf{h}_{k-1}\textbf{h}_k\textbf{h}_1$ or $\textbf{h}_{1}\textbf{h}_{k}\textbf{h}_{k-1}\dots \textbf{h}_2\textbf{h}_1$. In the former case we say that the orientation of the sequence is  positive, in the latter one we say that it is  negative. 

Thus, let us choose $n/(7k)$ sequences with the same orientation. We subdivide $P'$ into $k-1$ connected sections, such that each of them contain at least $z = n/(8k^2)$ sequences of the same orientation and denote those sections as $Q_1,...,Q_{k-1}$. See Figure~2 for an example.

\begin{figure}[h!]
  \label{rysunek2}
  \centering
    \includegraphics[width=\textwidth]{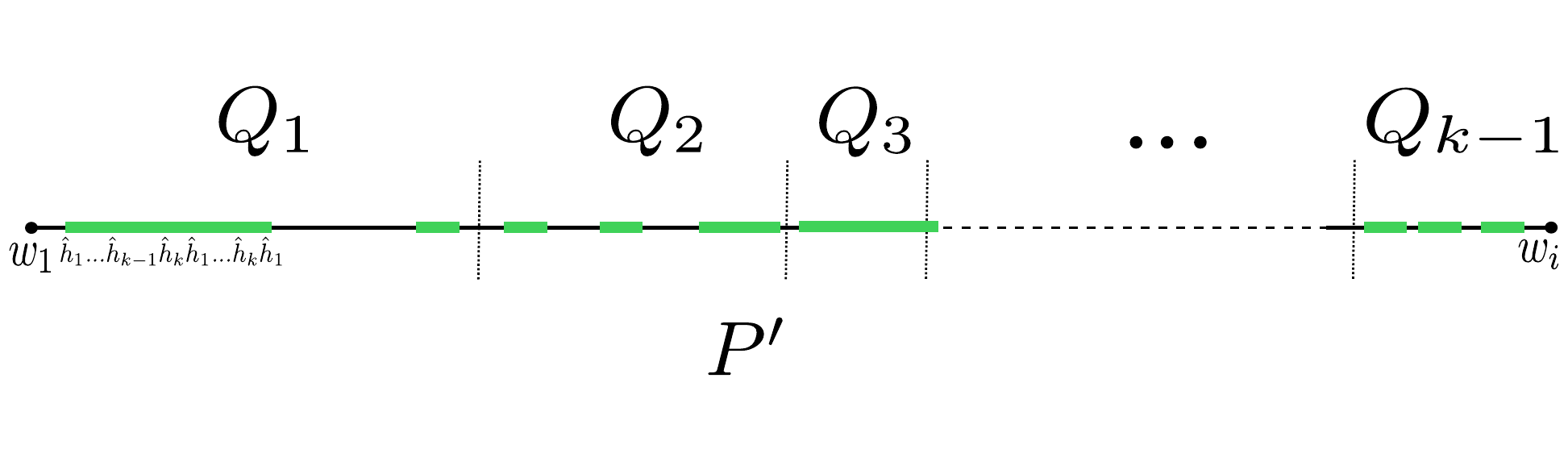}
      \caption{A path $P'$ divided into sections $Q_1,...,Q_{k-1}$. By $\hat{h}_i$ we denote that vertex is an element of the fiber above $\textbf{h}_i$. Green segments indicate sequences of vertices which belong to fibers given by order of vertices in $\textbf{H}_1$ (there could be more than one sequence in one segment).}
\end{figure}

Let $\vec{\textbf{H}}_1$ be a directed cycle created by orienting edges of $\textbf{H}_1$ in one direction. 
Let us recall that $\textbf{H}_2=\textbf{h}'_1\textbf{h}'_2\dots \textbf{h}'_{k}\textbf{h}'_1$ is the second Hamiltonian cycle in the base graph $\textbf{G}$ which does not share any 
edges with~$\textbf{H}_1$. In the definition below we treat each undirected edge of $\textbf{H}_2$ as a pair of edges with opposite orientations.

\begin{defin}
A directed path $P$ in $G$ is called $H_2\vec{H}_1$-alternating if it starts by an edge of $H_2$, ends with an edge of $\vec{H}_1$ and its edges belong alternatively to Hamilton cycle $H_2$ and directed cycle $\vec{H}_1$.  
\end{defin}

The proof of the following lemma can be found in the next section of this paper.

\begin{lemma}
\label{lemSwitch}
Let $H_1$ and $H_2$ be edge disjoint Hamilton cycles such that $H = H_1 \cup H_2$ is a non-bipartite graph, then for every pair of vertices $v,u \in H$ there exists a $H_2\vec{H}_1$-alternating path from $v$ to $u$.
\end{lemma}

Denote by $u$ the end of $P'$ we want to switch and let $\tilde{G}_x$ be a fiber which we want to switch $u$ onto. Without loss of generality we may assume that the end $u$ belongs to the fiber $\tilde G_z$ above some vertex $\textbf{z}$. 
By Lemma~\ref{lemSwitch} in $\textbf{G}$ there exist a $\textbf{H}_2\vec{\textbf{H}}_1$-alternating path $\textbf{R}=\textbf{zb}_1\textbf{a}_1\textbf{b}_2\textbf{a}_2,..,\textbf{b}_{\ell-1}\textbf{a}_\ell \textbf{b}_\ell \textbf{x}$ from $\textbf{z}$ to $\textbf{x}$.

 Therefore we generate an edge from $u$ to a vertex $b^1 \in \tilde{G}_{b_1}$. The probability that adjacent vertex belongs to $Q_1$ equals $1/9k^2$. If $b^1$ is not a part of $Q_1$ then we stop, otherwise we use it as the pivot in P\'osa transformation that would change $P'$ into a path $P'_1$ which ends at a vertex $a^1 \in  \tilde{G}_{a_1}$. Then we continue the transformations in the same manner as in the first step. We reveal an edge from $a^1$ to a vertex $b^2 \in \tilde{G}_{b_2}$. Again with probability $1/9k^2$ vertex $b^2$ belongs to $Q_2$. Next we use $b^2$ as the pivot in P\'osa transformation that would change $P'_1$ into a path $P'_2$ which ends at a vertex $a^2$ from fiber above $\textbf{a}_2$. We apply the same operations until we get to a vertex $w \in \tilde{G}_x$. Notice since pivot for path $P_i$ is closer on a path $P_i$ to the vertex $w_i$ than the pivot used for path $P_{i-1}$ the transformation does not change the orientation of the sequences in the $Q_{i+1},..Q_{k-1}$. See Figure~3 for an example.
 
\begin{figure}[h!]
  \label{rysunek3}
  \centering
    \includegraphics[width=\textwidth]{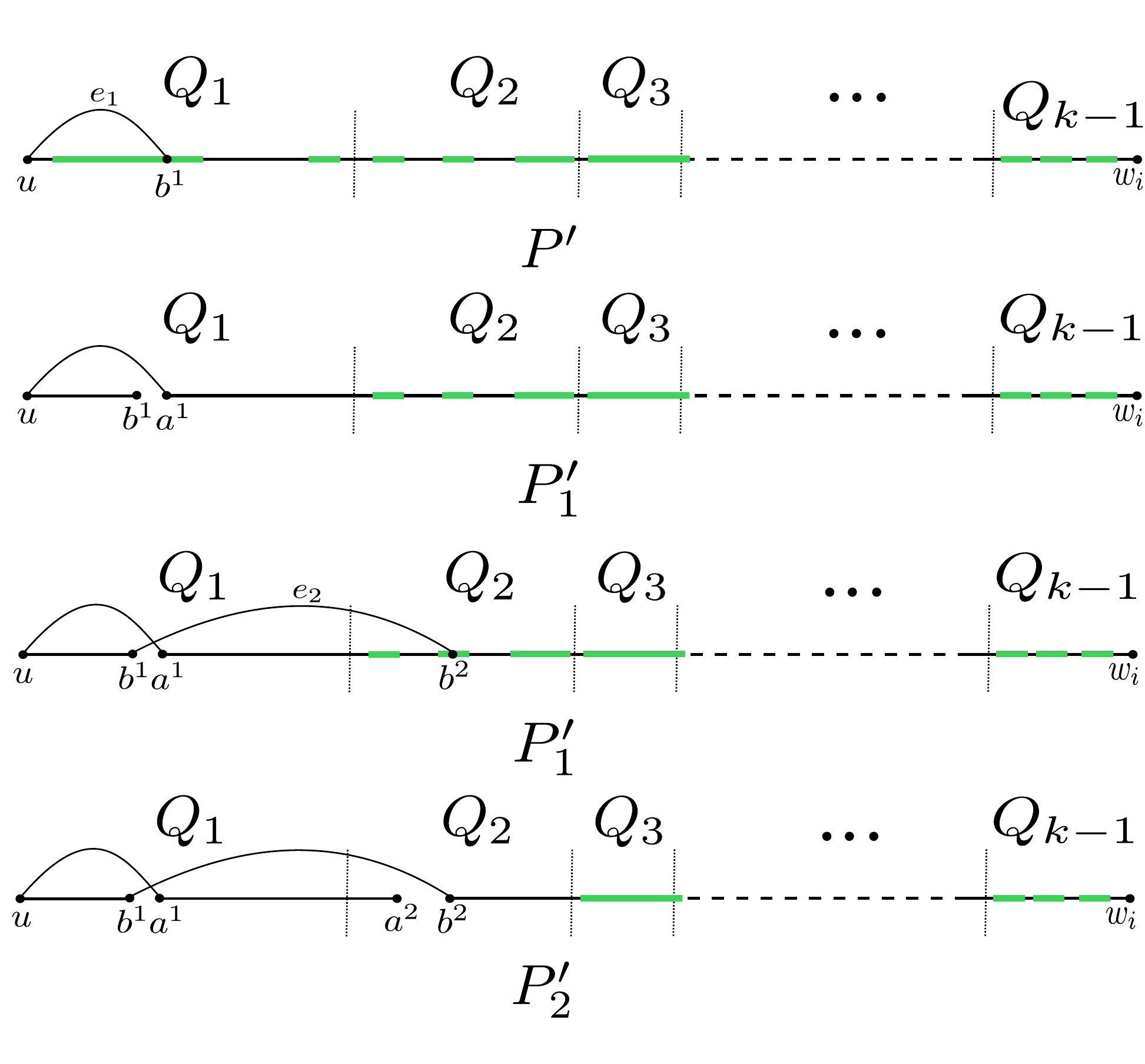}
      \caption{Two steps of the process of switching the end of path $P'$ onto desired fiber. Green sections indicates positively oriented sequences of vertices from fibers above consecutive vertices of the cycle $\textbf{H}_1$. The vertex $u$ is the end of path $P'$. The edge $e_1$ connects vertex $u$ with vertex $b^1$ from fiber $\tilde{G}_{b_1}$. The edge $e_2$ connects vertex $a^1$ with vertex $b^2$ from fiber $\tilde{G}_{b_2}$.}
\end{figure}

  Note that since the length of a $\textbf{H}_2\vec{\textbf{H}}_1$-alternating path is bounded by $2(k-1)$ during the process we have to generate at most $k-1$ edges (those which belongs to $\tilde{\textbf{H}}_2$). Hence, with probability at least $(9k^2)^{-k}>\xi>0$ we can move a given vertex from $S_1$, from any fiber to the designated fiber above vertex~$\textbf{x}$. The same analysis can be repeated in respect to the second path $P''$ and vertices from $S_2$ which we would like to place on fiber $\tilde{G}_y$. Since $|S_1|, |S_2|\ge n^{3/5}\log^2 n$, from Chernoff's bound we get that with probability at least $1-\exp(-n^{3/5})=1-o(1/\log n)$  we can successfully switch at least $n^{3/5}$ of them. Note that in this process we deactivated at most $2|\textbf{G}|n^{3/5}\log^2 n<n^{4/5}$ new vertices.

\smallskip

{\bf Phase 7.} 
Since $S'_1$ and $S'_2$ belong to different fibers which correspond to adjacent vertices from $G$ the probability that there are no edges
between $S'_1$ and $S'_2$ is bounded from above by 
\begin{align*}
\left(\frac{n-|S'_1|-|D|}{n-|D|} \right)^{|S'_2|}  \leq \Big(\frac{n-|S'_1|}{n}\Big)^{|S'_2|}
& \le \exp\Big(-\frac{|S'_1||S'_2|}{2n}\Big) \\
&  \le \exp(-n^{1/6}/2) = o(1/\log n)\, .
\end{align*}
Clearly, in the last phase we deactivated at most $|S'_1|\le n^{4/5}$ vertices.

This completes the analysis of the algorithm and the proof 
of Theorem~\ref{thm:main}.\hfil\qed

\section{Proof of the Lemma}

In this section we present the proof of the Lemma \ref{lemSwitch}.
In fact we shall show a slightly more general result.

\begin{lemma}
Let $\vec{H}_1$  be a directed Hamilton cycle and $H_2$ be a connected
$d$-regular graph, edge disjoint with $H_1$ which is such that $H = H_1 \cup H_2$ is not a bipartite graph. Then for every pair of vertices $v,u \in H$ there exists a $H_2\vec{H}_1$-alternating path from $v$ to $u$.
\end{lemma}

\begin{proof}
For the purpose of the proof let color edges of $\vec{H}_1$ \textit{red} and edges of $H_2$ \textit{blue}. We proceed in following way: starting from vertex $v$ we build $H_2\vec{H}_1$-alternating paths to other vertices. We mark vertices in $H$ \textit{red} if we leave this vertex by a red edge and respectively \textit{blue} we leave this vertex by a blue edge (vertices that admits both colors are denoted as red-blue). Notice that it means that vertex $v$ gets blue color. 
Notice that $H_2\vec{H}_1$-alternating path is now equivalent to a BlueRed-alternating path, starting with a blue edge and ending with a red edge (that is, ending in a blue or red-blue vertex). 

Denote by $N$ the set of vertices that are not reached 
from vertex $v$ by $H_2\vec{H}_1$ alternating path (that did not get any color) and by $R$, $B$, $RB$ the sets of vertices of appropriate colors.
Let us make the following observations.

\begin{enumerate}
\item[(i)] There are no directed red edges $\{xy\}$ from $x \in R \cup RB$ to $y \in N$, because that would imply $y \in B$. The same argument shows that there are no blue edges between $RB$ and $N$. 

\item[(ii)] There is in total at most one directed red edge $\{xy\}$ coming from  $N$ to $B \cup RB$, or coming from sets $B$ to $RB$, or within the set $B$. Indeed,  apart from the starting vertex $v$ in order to color vertex blue we have to first reach it by a red edge 
(coming to it from a red vertex). 

\item[(iii)] There are no blue edges $\{xy\}$ between $x \in RB$ and $y \in B$, since in this case we would be able to reach $x$ by red edge, use blue edge to get to $y$ and then leave $y$ by red edge, which results in  $y \in RB$. The same argument shows that there are no red edges directed from $RB$ to $R$, no blue edges inside $B$, and no red edges inside $R$. 
\end{enumerate}

Figure $4$ shows all the possible  edges which can occur 
between sets $R$, $B$, $RB$ and $N$.

\begin{figure}[h!]
\label{LemmaRys}
  \centering
    \includegraphics[width=\textwidth]{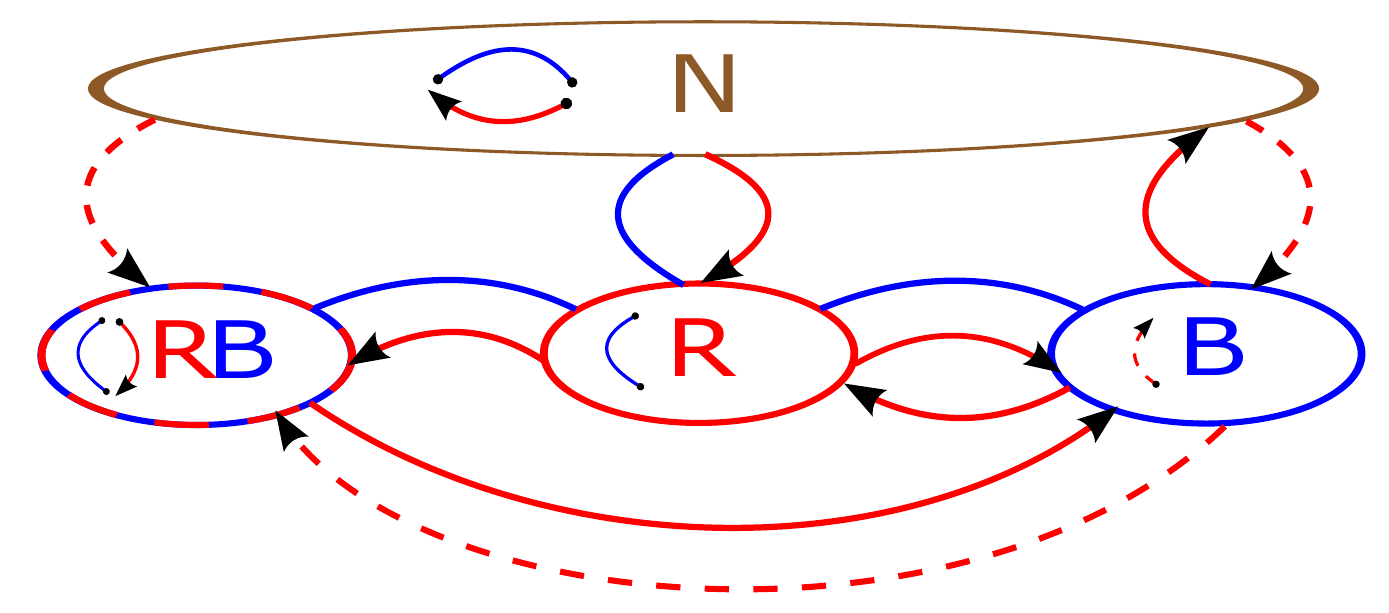}
      \caption{Diagram of edges between vertices in $H = H_1 \cup H_2$ created by following $H_2\vec{H}_1$-alternating paths from some  vertex $v \in H$. Edges of Hamilton cycle $H_1$ are directed and coloured red, and edges of Hamilton cycle $H_2$ are coloured blue.
			There are at most one red edge in total in places denoted by red dashed arrows.}
\end{figure}

Denote by $|X \rightarrow Y|$  the number of red edges coming from the set $X$ to the set $Y$. 
Since the red edges form the directed Hamilton cycle $\vec{H}_1$
in $G$ there is exactly one red edge coming out and one red edge coming in to every vertex in $G$. Thus, we can estimate the size of 
$R$  counting red edges incoming to it. Therefore 
$$ |R| = |N\rightarrow R| + |B \rightarrow R|. $$
In a similar way (counting red edges outgoing from vertices), we have
$$ |B| = |B \rightarrow N| + |B \rightarrow R| + |B \rightarrow RB| + |B \rightarrow B| $$
The red edges form a directed Hamilton cycle, thus the number of red edges coming in to the set $N$ equals to the number of red edges coming out from the set $N$. Furthermore,  by (ii), we have 
$$|N\to B|+|N\to RB|+|B\to RB|+|B\to B|\le 1\,.$$
Hence
$$|B \rightarrow N| + |B \rightarrow RB| + |B \rightarrow B| \geq |N\rightarrow R|\,$$
which implies $|B| \geq |R|$.

Each blue and red vertex is incident to exactly $d$ blue edges. Notice that all blue edges that are leaving $B$ go to the vertices in $R$, which, together with $|B| \geq |R|$, implies that $|B| = |R|$. 
Let us consider now three possible sizes of set $|B|=|R|$.

\medskip

{\sl Case 1.}\quad $|B|=|R|=0$.

\smallskip 

Then all vertices of $G$ belong to $RB\cup N$, and since 
both $H_1$ and $H_2$ are connected graphs we must have $N=\emptyset$.
Consequently, all vertices are colored with both colors and the
assertion follows.

\smallskip

{\sl Case 2.}\quad $0<|B|+|R|<|G|$.

\smallskip 

Then the blue edges between $R$ and $B$ induce the 
$d$-regular subgraph of $H_2$, which is clearly a
component of $H_2$. This contradicts the fact that 
$H_2$ is connected. 

\smallskip

{\sl Case 3.}\quad $|B|+|R|=|G|$.

\smallskip 

In this case  $G$ is a bipartite graph in which red vertices form one part of the partition and blue vertices form the other part, which contradicts the assumption of the lemma. 

This concludes the proof of the lemma.
\end{proof}

\bibliographystyle{plain}


\begin{thebibliography}{9}

\bibitem{pierwsza} {A.~Amit and N.~Linial},
{\em Random graph coverings \texttt{I}: General theory and graph connectivity}, {Combinatorica}, {\bf 22} (2002), 1--18.

\bibitem{trzecia} {A.~Amit, N.~Linial,  and J.~Matousek},
{\em Random lifts of graphs: Independence and chromatic number},
{Random Structures \& Algorithms}, {\bf 20}, (2002), 1--22.

\bibitem{druga} A.~Amit and N.~Linial,
{\em Random lifts of graphs: edge expansion},
{Combinatorics, Probability \& Computing}, {\bf 15}
(2006), 317--332.

\bibitem{frieze}
{K.~Burgin, P.~Chebolu, C.~Cooper, and A.M.~Frieze},
{\em Hamilton cycles in random lifts of graphs},
{European Journal of Combinatorics}, {\bf 27} (2006), 1282--1293.

\bibitem{fr2} {P.~Chebolu,  and A.~Frieze},
{\em Hamilton cycles in random lifts of directed graphs},
{SIAM Journal of Discrete Mathematics}, {\bf 22} (2008), {520--540}.

\bibitem{feller}
 {W.~Feller},
An Introduction to Probability Theory and Its Applications, Vol. 1,
3rd edition, John Wiley \&\ Sons, Inc., New York-London-Sydney 1968.

\bibitem{harris}
 {T.E. Harris},
 {The theory of branching processes}, 
 Die Grundlehren der mathematischen Wissenschaften,
 Springer, Berlin, 
 1963.



\bibitem{JLR} {S.~Janson, T.~\L uczak, and A.~Ruci\'nski},
Random Graphs,  Wiley, New York, 2000.

\bibitem{KrSud} {M.~Krivelevich, B.~Sudakov},
{\em Sparse pseudo-random graphs are Hamiltonian},
J. Graph Theory {\bf 42} (2003),  17--33.

\bibitem{nati2}  {N.~Linial},
     {Random lifts of graphs}, presentation, summer '05, {\tt
     http://www.cs.huji.ac.il/\~{}nati/PAPERS/lifts\underline{\ }talk.pdf}

\bibitem{czwarta} {N.~Linial and E.~Rozenman},
{\em Random lifts of graphs: perfect matchings},
{Combinatorica}, {\bf 25} (2005), 407--424.

\bibitem{Posa}  {L. P\'{o}sa}, {\em Hamiltonian circuits in random graphs}, {Discrete Mathematics}, {\bf 14} (1976), 359--364.

\bibitem{Witkowski}  {M. Witkowski},
{\em Random lifts of graphs are highly connected},
{Electronic Journal of Combinatorics} {\bf 20(2)} (2013).

\end{thebibliography}

\end{document}